\documentclass[a4paper, reqno]{amsart}
\usepackage{amsgen,amsmath,amstext,amsbsy,amsopn,amsfonts,amssymb}
\usepackage{amsmath,amssymb,amsfonts,amsthm,enumerate}
\usepackage{amsthm}
\usepackage[latin1]{inputenc} 
\usepackage{color}
\usepackage[pagebackref]{hyperref}
\usepackage[alphabetic]{amsrefs}
\usepackage{xcolor}
\usepackage{tikz-cd}

\usepackage{rotating}

\newtheorem{theorem}{Theorem}
\newtheorem{lemma}[theorem]{Lemma}

\newtheorem{proposition}[theorem]{Proposition}
\newtheorem{obs}[theorem]{Observation} \newtheorem{defi}[theorem]{Definition}

\newenvironment{definition}{\begin{defi}\rm}{\end{defi}}
\newtheorem{exa}[theorem]{Example}

\newtheorem{rem}[theorem]{Remark}
\newenvironment{remark}{\begin{rem}\rm}{\end{rem}}
\newtheorem{rems}[theorem]{Remarks}

\newtheorem*{acknowlegment}{Acknowlegment}

\def\H{\mathcal H}

\def\K{\mathcal K}
\def\P{\mathcal P}

\def\NN{{\mathbf N}}
\def\ZZ{{\mathbf Z}}
\def\CCC{{\mathbf C}}
\def\RRR{{\mathbf R}}
\def\QQ{\mathbf Q}

\def\AA{{\mathbf A}}
\def\RR+{{\mathbf R}^*}

\def\Q_p{{\mathbf Q}_p}

\def\HH{\mathbf H}

\def\ind{{\rm Ind}}

\def\eps{\varepsilon}

\def\la{\lambda}

\def\vfi{\varphi}

\def\ind{{\rm Ind}}

\newcommand{\Tr}{\operatorname{Tr}}

\newcommand{\Hom}{\operatorname{Hom}}

\def\tout{\qquad\text{for all}\quad}

\def\Q{\mathbf Q}

\newtheorem{thmx}{Theorem}
 
\newtheorem{corx}[thmx]{Corollary}
\begin{document}

\title[CCR for general rings]{ Canonical Commutation Relations: A quick proof of the Stone-von Neumann theorem 
 and an extension to general rings}

\address{Bachir Bekka \\ Univ Rennes \\ CNRS, IRMAR--UMR 6625\\
Campus Beaulieu\\ F-35042  Rennes Cedex\\
 France}
 \date{\today}
\email{bachir.bekka@univ-rennes1.fr}
\author{Bachir Bekka}

\thanks{The author acknowledges the support  by the ANR (French Agence Nationale de la Recherche)
through the project Labex Lebesgue (ANR-11-LABX-0020-01) .}
\begin{abstract}
Let $R$ be  a (not necessary commutative)  ring with unit,
$d\geq 1$ an integer, and $\lambda$ a unitary character
of the additive group $(R,+).$
A pair  $(U,V)$ of unitary representations $U$ and $V$ 
of  $R^d$  on a Hilbert space $\H$  is said
to satisfy the  canonical commutation relations (relative to $\lambda$) if
$$U(a) V(b)= \lambda(a\cdot b)V(b) U(a)$$ for all $a=(a_1, \dots, a_d), b= (b_1, \dots, b_d)\in R^d$,
where $a\cdot b= \sum_{k=1}^d a_k b_k.$
We give a new and quick proof of the classical Stone von Neumann Theorem
about the   essential uniqueness of such a pair  in the case  where $R$ is a local field (e.g. $R= \mathbf{R}$).
Our methods allow us to give the following extension of this result 
to a general  locally compact ring $R$.  For a unitary representation  
$U$ of $R^d$ on a Hilbert space $\H, $ define the inflation $U^{(\infty)}$ of $U$ 
as the unitary representation $U^{(\infty)}$  of $R^d$ on the
Hilbert space  $ \H^{(\infty)}=\oplus_{i\in \NN} \H$ given by  
$U^{(\infty)}(a) (\oplus_{i\in \NN} \xi_i) =  \oplus_{i\in \NN} U(a) \xi_i.$

Let $(U_1, V_1), (U_2, V_2)$  be two pairs of unitary representations
of $R^d$ on corresponding Hilbert spaces $\H_1, \H_2$ satisfying the  canonical commutation relations
 (relative to $\lambda$).
 Provided  that $\la$ satisfies a mild faithful condition, 
we show  that the  inflations $(U_1^{(\infty)}, V_1^{(\infty)}), (U_2^{(\infty)}, V_2^{(\infty)})$
 are approximately equivalent, that is, there exists a sequence
$(\Phi_n)_n$ of unitary isomorphisms $\Phi_n: \H_1^{(\infty)}\to \H_2^{(\infty)}$  such  that
$$\lim_{n} \Vert U_2^{(\infty)}(a) - \Phi_n U_1^{(\infty)}(a) \Phi_n^{*}\Vert=0 
\quad \text{and} \quad \lim_{n} \Vert V_2^{(\infty)}(b) - \Phi_n V_1^{(\infty)}(b) \Phi_n^{*}\Vert=0,$$
uniformly on compact subsets of $R^d.$

\end{abstract}
\subjclass[2000]{22D10; 	43A25; 22B05 }
\maketitle
\section{Introduction}
Heisenberg's original canonical commutation relations (see \cite{Heisenberg2}) for a system with $d$ degrees of freedom
involve    position operators $q_1, \dots, q_d$ and   momentum operators
$p_1, \dots, p_d$,  which are
unbounded operators  on a complex Hilbert space 
satisfying the relations
$$p_k q_l-q_lp_k = \dfrac{h}{2\pi i} \delta_{kl}I,$$
where $h$ is Planck's constant.
Considering the unitary  operators 
$$U(t) = \exp(\sum_{k=1}^d  \dfrac{2\pi i}{h}t_k p_k) \quad\text{and}\quad
V(s) = \exp(\sum_{l=1}^d  \dfrac{2\pi i}{h} s_l q_l)$$
 on a Hilbert space $\H$ 
for $t=(t_1, \dots, t_d)$ and $s= (s_1, \dots, s_d)$ in $ \mathbf{R}^d,$
H. Weyl (see \cite[p. 26]{Weyl}) gave the canonical commutation relations the following form:
\begin{equation}
U(t) V(s)= e^{\frac{2\pi i}{h} t\cdot s}V(s) U(t) \tag*{(CCR)}\label{(CCR)}
\end{equation}
where  $t\cdot s= \sum_{k=1}^d t_ks_k.$
A solution of  \ref{(CCR)} is given by the so-called Schr\"odinger representation
and the classical Stone-von Neumann Theorem asserts that  this solution is  essentially unique 
(see below for  precise statements).
 We are going to give  a quick proof of this   theorem and to   extend it   to more general rings. 
 
 Given a topological group $G$ and a complex Hilbert space $\H,$ recall that a  unitary representation $\pi$
 of $G$ on $\H$ is a continuous homomorphism $\pi: G\to U(\H)$, where   $U(\H)$ is the group of unitary operators
of $\H$ equipped with the weak (or strong) operator topology.
For a finite or infinite countable cardinal  $n\in \{1, 2, \hdots, \infty \}$, we denote by
$$\H^{(n)}= \oplus_{i\in\{1,2, \hdots, n\}} \H$$
  the Hilbert space direct sum of $n$ many copies of $\H$
and  by $\pi^{(n)}$ the corresponding \emph{inflation} of $\pi;$
so,  $\pi^{(n)}: G \to U(\H^{(n)})$ is  the unitary representation of $G$ on $\H^{(n)}$ defined by 
$$\pi^{(n)}(g)= \oplus_{i\in\{1,2, \hdots, n\}} \pi(g) \tout g\in G.$$

 Let  $R$ be  a   locally compact  ring with unit; notice  that we do not assume that $R$ is commutative.
Fix   a unitary character  $\lambda$ of   the additive group $(R, +)$, that is,  a one dimensional unitary representation  of $(R,+)$.
We say that  a pair $(U,V)$ of unitary representations $U$ and $V$ 
of   $R^d$  on a  Hilbert space $\H$ satisfies the  canonical commutation relations (relative to $\lambda$) if
\begin{equation}
U(a) V(b)= \lambda(a\cdot b)V(b) U(a) \tag*{$\rm{(CCR)}_\lambda$} \label{CCR-la}
\end{equation} 
for all $a=(a_1, \dots, a_d)$ and $b= (b_1, \dots, b_d)$ in $ R^d,$
where $a\cdot b= \sum_{k=1}^d a_kb_k.$

It is quickly checked that a solution of \ref{CCR-la}, called \emph{Schr\"odinger pair}
(or \emph{Schr\"odinger representation}), is given by the pair $(U_{\text{Schr}}, V_{\text{Schr}})$    defined on $ L^2(R^d, \mu)$ by 
$$
(U_{\text{Schr}}(a)f)(x)= f(x+a) \quad\text{and} \quad (V_{\text{Schr}}(b) \xi)(x)= \lambda(x\cdot b) f(x) \quad\text{for} \quad f \in L^2(R^d, \mu),
$$
where $\mu$ is a Haar measure on $(R^d,+).$

There is an obvious equivalence relation on pairs of unitary representations
of $R^d$: two such pairs  $(U_1,V_1)$ and $(U_2, V_2)$ acting on  respective   Hilbert  spaces $\H_1$
and $\H_2$ are \emph{equivalent} if there exists a unitary  isomorphism $\Phi: \H_1\to \H_2$
such that
$$\Phi U_1(a) \Phi^*= U_2(a)\quad\text{and}\quad\Phi V_1(b)\Phi^*= V_2(b) $$
for all $a, b\in R^d.$

Let $\widehat{R}$ be the dual group of $R,$  that is, the group of all unitary characters of $(R,+).$
Given $\la$ in $\widehat{R},$ 
properties of the  continuous homomorphism
 $$
  \nabla_\lambda: R\mapsto \widehat{R},\  a\mapsto (x\mapsto \la(x\cdot a))
  $$
  will play an important role in the sequel.
\begin{thmx}
\label{Theo2}
(\textbf{Stone-von Neumann Theorem})
Let  $R$ be  a  unital   second-countable locally compact ring and let  $\lambda\in \widehat{R}.$
We assume that
\begin{subequations}
\begin{align}
 &\bullet \lambda(ab)= \lambda(ba) \quad\text{for all} \quad a,b\in R
 \tag*{(Sym)} \label{Sym} \\
&\bullet
 \nabla_\lambda: R\mapsto \widehat{R} \quad\text{is an isomorphism.}
\tag*{(Isom)}\label{(Iso)}  
\end{align}
\end{subequations}
Let $(U,V)$ be a pair of unitary representations  of  $R^d$ on a separable Hilbert space satisfying  \ref{CCR-la}. 
Then  the inflation  $(U^{(\infty)},V^{(\infty)})$ of $(U,V)$  is  equivalent to  the inflation  $(U_{\text{Schr}}^{(\infty)}, V_{\text{Schr}}^{(\infty)})$ of  the Schr\"odinger pair.
\end{thmx}
The following more familiar form of the Stone-von Neumann Theorem is an immediate 
consequence of Theorem~\ref{Theo2}.
\begin{corx}
\label{Cor} (\textbf{Stone-von Neumann Theorem-bis})
Let $R$ and $\la$ be as in Theorem~\ref{Theo2}. 
Let $(U,V)$ be a pair of unitary representations 
of  $R^d$ on a separable Hilbert space $\H$ satisfying  \ref{CCR-la}. 
Then there exist a subset $J$ of $\NN$,
an orthogonal decomposition $\H= \oplus_{j\in J}{\H_j}$ of $\H$ into closed 
$U(R^d)$- and $V(R^d)$-invariant subspaces $\H_j$   as well as  
 unitary isomorphisms $\Phi_j:  L^2(R^d)\to \H_j$
 such that 
 $$U(t)|_{\H_j} = \Phi_j U_{\text{Schr}}(a) \Phi_j^{*} \quad\text{and}\quad
 V(s)|_{\H_j} = \Phi_j V_{\text{Schr}}(b) \Phi_j^{*},$$
 for all $a,b \in R^d$ and all $j\in J.$
\end{corx}

Here are examples of rings $R$   for which there exists  $\la\in \widehat{R}$
such that  Conditions~\ref{Sym}  and \ref{(Iso)} above hold (see Section~\ref{S:Examples} below;
for some of these rings, even every non trivial $\la$ satisfies these two conditions):
\begin{enumerate}
\item $R=\mathbf{k}$ is  a  local field $\mathbf{k},$ that is,  a locally compact non discrete field;
examples include $\mathbf{k}=\mathbf{R}$ and  $\mathbf{k}=\mathbf{Q}_p$ (the fields of $p$-adic 
numbers).
\item $R= \mathbf{A}$  is the ring of ad\`eles of $\QQ.$
\item $R=M_n(\mathbf{k})$ is the ring of $n\times n$-matrices over a  local field $\mathbf{k}.$
\item $R= \ZZ/n\ZZ$ is the  ring of integers modulo $n\geq 1.$
\end{enumerate}

A few words about the  proofs of the Stone-von Neumann Theorem are in order.
The original proof in the case of $R= \RRR$ is due  to von Neumann (\cite{VN});  his methods were
  later extended to cover other local fields  by Segal  \cite{Segal} and  Weil \cite{Weil}.
  In all these approaches, the  pair $(U,V)$ acting on $\H$ is ``integrated "
  to a representation of a certain algebra of functions on $R^d\times R^d$ 
  by an algebra of operators on $\H$ and
  one shows that  this latter  algebra contains sufficiently many rank-one operators.
  Mackey stated and proved a far-reaching generalization of  this theorem (\cite{Mackey-SVN});
  his proof uses a different approach which amounts to the classification  of so-called transitive systems of imprimitivity based on $R^d$  (in fact,  on an arbitrary second countable locally compact group).
 In comparison, our proof is more elementary; the basic idea is an imitation of  a trick,  called  Fell's  absorption principle (see  Proposition~\ref{Pro-Reg} below), and  the only sophisticated  tool we   use  is  Plancherel's theorem about the Fourier transform of  the abelian group $R^d.$
 Our approach confirms somehow the observation made by Howe~\cite[p.827]{Howe} 
 that the Stone-von Neumann and the Plancherel Theorems are equivalent.

Rings for which there exists no $\la\in \widehat{R}$
such that  Condition \ref{(Iso)} holds include the ring of integers $\ZZ$
and, more generally, every countable infinite ring with the discrete topology.
Theorem~\ref{Theo2} fails for such  rings.
Nevertheless, our approach  allows us   to prove
for them the following weaker version of the Stone-von Neumann Theorem.

We need to define a  second  equivalence relation on pairs of unitary representations
of $R^d$; this definition   is tailored after a similar  notion of approximately equivalent  representations of $C^*$-algebras  (see \cite[Definition 41.1]{Conway}).

\begin{definition}
\label{Def-AE}
 Two pairs  $(U_1,V_1)$ and $(U_2, V_2)$  of unitary representations
of $R^d$ acting on  respective   Hilbert  spaces $\H_1$
and $\H_2$ are  \emph{approximately equivalent} if there exists
 a sequence  $(\Phi_n)_{n}$ of unitary  isomorphisms $\Phi_n: \H_1\to \H_2$
 such that, for every compact subset $K$ of $R^d,$ we have
$$ \lim_{n\to \infty} \sup_{a\in K}\Vert U_2(a)-\Phi_n U_1(a)\Phi^*_n\Vert=0 \quad\text{and}\quad  
\lim_{n\to \infty}  \sup_{b\in K}
\Vert  V_2(b)- \Phi_nV_1(b)\Phi^*_n\Vert=0.$$
 \end{definition}

\begin{thmx}
\label{Theo1}
 Let  $R$ be  a  unital   second-countable locally compact ring. Let  $\lambda\in \widehat{R}$
  satisfying Condition~\ref{Sym} from Theorem~\ref{Theo2} and such that
\begin{equation}
\text{there is  no non-zero  two-sided ideal of} \quad R \quad\text{contained in} \quad \ker \la.
\tag*{(Faith)}\label{Dens}
\end{equation}
For an integer $d\geq 1,$ let $(U,V)$ be a pair of unitary representations 
of  $R^d$ on a separable Hilbert space $\H$ satisfying the commutation relation  \ref{CCR-la}.
 Then  $(U^{(\infty)},V^{(\infty)})$   is approximately equivalent
to  $(U_{\text{Schr}}^{(\infty)}, V_{\text{Schr}}^{(\infty)})$.

\end{thmx}
Condition~\ref{Dens} is a weak faithfulness condition on $\la;$  provided $\la$ is non trivial, it holds for instance  when  $R$ is a simple ring.

As is nowadays well-known, an appropriate way to formulate the commutation relations 
is by means   of unitary representations of  a  Heisenberg group  associated to $R$
and $d$ (see  \cite[Chap.1, \S 5]{Folland} in the case $R= \mathbf{R}$).
We will recall the definition of this group and rephrase our results in terms of its
representations in Section~\ref{Heisenberg}.

\section{Proof of Theorem~\ref{Theo2} and   Theorem~\ref{Theo1}}
\label{S:ProofTheo}

\vskip.3cm
Let $\la$ be an arbitrary unitary character of $R.$
We  will use in a crucial way the following   pair $(U_{\rm reg}, V_{\rm reg})$  of unitary representations of $R^d$
satisfying \ref{CCR-la}, which we call the \emph{regular  pair} for reasons we will explain below (Remark~\ref{Rem-Reg}).
Set 
$$S:= R^d\times R^d$$
 and  consider the Hilbert space $L^2(S)= L^2(S, \mu\otimes \mu),$
where $\mu$ is a fixed Haar measure on $R^d.$
Let $U_{\rm reg}$ and  $V_{\rm reg}$ be defined on $L^2(S)$ by 
$$
(U_{\rm{reg}}(a)\xi)(x,y)= \xi(x+a,y) \quad\text{and} \quad (V_{\rm{reg}}(b) \xi)(x)= \lambda(x\cdot b)\xi(x,y+b)
$$
for all $\xi \in L^2(S)$ and $a,b, x, y\in R^d.$
One checks immediately that $(U_{\rm reg}, V_{\rm reg})$ satisfies the relation  \ref{CCR-la}.

Fix a   pair  $(U,V)$  of unitary representations 
of  $R^d$ on a  separable Hilbert space $\H$ satisfying the commutation relation  \ref{CCR-la}
Let $m,n\in \{1, 2, \hdots, \infty\}$
be the dimensions  of  the Hilbert spaces $L^2(S)$ and $\H$ respectively. 
The Hilbert space  $L^2(S, \H)= L^2(S, \H, \mu\otimes \mu)$
of (equivalence classes of) square-integrable measurable maps $S \to \H$ 
is isomorphic to $\H^{(m)}$ as well as to $L^2(S)^{(n)}$;
  the inflations $(U^{(m)}, V^{(m)})$ of $(U,V)$ and   $(U_{\rm reg}^{(n)}, V_{\rm reg}^{(n)})$
of  $(U_{\rm reg}, V_{\rm reg})$ are realized on $L^2(S, \H)$ by the formulae
$$
(U^{(m)}(a)F)(x,y)= U(a)(F(x,y) \quad\text{and} \quad (V^{(m)}(b) F)(x)=V(b)F(x,y)
$$
and 
$$
(U_{\rm{reg}}^{(n)}(a)F)(x,y)= F(x+a,y) \quad\text{and} \quad (V_{\rm{reg}}^{(n)}(b) F)(x)= \lambda(x\cdot b)F(x,y+b)
$$
for all $F \in L^2(S, \H)$ and $a,b, x, y\in R^d.$

We associate to $(U,V)$ another pair  $(\widetilde{U},\widetilde{V})$   of unitary representations  of  $R^d$ on   $L^2(S, \H)$
 defined by 
$$
(\widetilde{U}(a)F)(x,y)= U(a)(F(x+a,y)) \quad\text{and} \quad(\widetilde{V}(b)F)(x,y)= V(b)(F(x,y+b))
$$
for all $F \in L^2(S, \H)$ and $a,b, x, y\in R^d.$ Again, one checks immediately that $(\widetilde{U},\widetilde{V})$  satisfies the relation  \ref{CCR-la}.

So, we have three pairs of of unitary representations  of  $R^d$ acting on   $L^2(S, \H)$
and satisfying  \ref{CCR-la}: $(U^{(m)}, V^{(m)})$, $(\widetilde{U},\widetilde{V}),$
and $(U_{\rm reg}^{(n)}, V_{\rm reg}^{(n)})$; observe that, for fixed $n$,
the last one is independent of $(U,V).$

The following  proposition is the crucial tool we will use
in  the proofs of  both Theorem~\ref{Theo2} and Theorem~\ref{Theo1};
 Item (i) is reminiscent of a property of the regular representation of a group (see \cite[Corollary E.2.6]{BHV}),
often called Fell's  absorption principle (for the precise relationship with this principle,  see Section~\ref{Heisenberg}).
\begin{proposition}
\label{Pro-Reg}
\item[(i)] The pair $(\widetilde{U},\widetilde{V})$ is equivalent to the pair  $(U_{\rm reg}^{(n)}, V_{\rm reg}^{(n)})$.
\item[(ii)]   The pair $(U^{(m)}, V^{(m)})$ 
is equivalent to  the pair $(\overline{U^{m}}, \overline{V^{m}})$ given  on  $L^2(S, \H)$ by 
$$
(\overline{U^{m}}(a)F)(x,y)= \la(a\cdot x)U(a)(F(x,y)) \quad\text{and} \quad(\overline{V^{m}}(b)F)(x,y)= \la(y\cdot b)V(b)(F(x,y))
$$
for all $F \in L^2(S, \H)$ and $a,b, x, y\in R^d.$

\end{proposition}
\begin{proof}
(i)  Let $\Phi: L^2(S, \H)\to L^2(S, \H)$ be the  unitary isomorphism  defined by
$$
\Phi(F)(x,y)= U(-x)V(-y) (F(x,y)) \tout F \in L^2(S, \H),  (x, y)\in S.
$$ 
 For $a,b\in R^d,$ we have
$$
\begin{aligned}
\left((\widetilde{U}(a) \Phi)(F)\right)(x,y)&= U(a)(\Phi(F)(x+a,y))\\
&=U(a)U(-x-a) V(-y) (F(x+a,y))\\
&= U(-x)V(-y)(F(x+a,y))\\
&= \Phi(U_{\rm{reg}}^{(n)}(a)F)(x,y))\\
\end{aligned}
$$
and, similarly,
$$
\begin{aligned}
\left((\widetilde{V}(b) \Phi)(F)\right)(x,y)&= \Phi(V_{\rm{reg}}^{(n)}(b)F)(x,y)).
\end{aligned}
$$
 
 \noindent
(ii)  Let $\Psi: L^2(S, \H)\to L^2(S, \H)$ be the  unitary isomorphism  defined by
$$
\Psi(F)(x,y)= V(-x)U(y) (F(x,y)) \tout F \in L^2(S, \H),  (x, y)\in S.
$$ 
For $a,b\in R^d,$ we have
$$
\begin{aligned}
 \Psi(U^{(m)}(a)F)(x,y)&= V(-x)U(y)U(a)(F(x,y))\\
&= V(-x)U(a)U(y)(F(x,y))\\
&= \la(a\cdot x) U(a)(V(-x)U(y)(F(x,y))\\
&= \left(\overline{U^{m}}(a)(\Psi(F))\right)(x,y)\\
\end{aligned}
$$
 and, similarly, 
 $$
\begin{aligned}
 \Psi(V^{(m)}(b)F)(x,y)&= \left(\overline{V^{m}}(b)(\Psi(F))\right)(x,y).
\end{aligned}
 $$
 \end{proof}

\subsection{Proof of Theorem~\ref{Theo2}}
\label{SS: ProofTheo2}
We assume that $\lambda$ satisfies Conditions ~\ref{Sym} and  \ref{(Iso)}.
 Since $\widehat{S}$ is,  in the canonical  way, topologically
isomorphic to  $\widehat{ R}^{2d},$   the map
$$
\nabla_\lambda^{(2d)}:S\mapsto \widehat{S},\, (a,b) \mapsto ((x,y) \mapsto \la(x\cdot a)\la(y\cdot b))
$$
is an  isomorphism of topological groups.
We can therefore identify the (inverse) Fourier transform  of $S$ with the map $\mathcal{F}:L^1(S)\to C(S)$
given by 
$$
\mathcal{F}(f)(x,y) = \int_S f(t,s) (\nabla_\lambda^{(2d)}(x,y))(t,s) d\mu(t) d\mu(s)=\int_S f(t,s)  \la(t\cdot x)\la(s\cdot y) d\mu(t) d\mu(s)
$$
for $f\in L^1(S).$    Upon properly normalizing $\mathcal{F},$ we  can assume by Plancherel's theorem (see  \cite[Chap.II, \S 1]{Bourbaki-TS})) that 
$\mathcal{F}$ extends from $L^1(S)\cap L^2(S)$ to a bijective isometry
$\mathcal{F}:L^2(S)\to L^2(S).$

We claim that  $\mathcal{F}$ induces an equivalence between the pairs  $(\overline{U^{m}}, \overline{V^{m}})$ 
and $(\widetilde{U},\widetilde{V})$ from Proposition~\ref{Pro-Reg}.
Indeed, let $\widetilde{\mathcal{F}}: L^2(S, \H)\to L^2(S,\H)$ be given by 
$$
\widetilde{\mathcal{F}}(F)(x,y)= \int_S  \la(t\cdot x)\la(s\cdot y) F(t,s) d\mu(t) d\mu(s)
$$
for $F\in L^2(S, \H);$ more precisely,  $\widetilde{\mathcal{F}}$ is defined on the linear span $X$ of 
$\{f\otimes v\, :\,  f\in L^2(S), v\in \H\}$, where $(f\otimes v)(x,y)= f(x,y) v$,
by 
$$
\widetilde{\mathcal{F}}(f\otimes v)=  \mathcal{F}(f) v;
$$
 since $\widetilde{\mathcal{F}}$  is  isometric on $X$ and has dense image, it  extends to a bijective isometry
from $L^2(S, \H)$ to  $L^2(S,\H)$.
 For $a,b\in R^d$ and $F\in L^2(S, \H),$ we have, using Condition \ref{Sym},
$$
\begin{aligned}
\left(\widetilde{U}(a)\widetilde{\mathcal{F}}(F)\right)(x,y)&= U(a)(\widetilde{\mathcal{F}}(F)(x+a,y))\\
&=U(a)\left(\int_S   \la(t\cdot (x+a))\la(s\cdot y) F(t,s)d\mu(t) d\mu(s)\right)\\
&= \int_S \la(t\cdot a)   \la(t\cdot x)\la(s\cdot y)) U(a)(F(t,s)) d\mu(t) d\mu(s)\\
&=  \int_S \la(a\cdot t)   \la(t\cdot x)\la(s\cdot y)) U(a)(F(t,s)) d\mu(t) d\mu(s)\\
&= \int_S  \la(t\cdot x) \la(s\cdot y)  (\overline{U^{m}}(a)F)(t,s) d\mu(t) d\mu(s)\\
&=\widetilde{\mathcal{F}} (\overline{U^{m}}(a)F)(x,y)\\
\end{aligned}
$$
and, similarly,
$$
\begin{aligned}
\left(\widetilde{V}(b)\widetilde{\mathcal{F}}(F)\right)(x,y)&= \widetilde{\mathcal{F}} (\overline{V^{m}}(b)F)(x,y).
\end{aligned}
$$
This concludes the proof of  Theorem~\ref{Theo2}; indeed,  denoting by $\sim$ the equivalence of pairs  of unitary representations of $R^d$ on separable Hilbert space $\H$ satisfying  \ref{CCR-la}, 
Proposition ~\ref{Pro-Reg} shows that
$$(U^{(m)}, V^{m})\sim (\overline{U^{m}}, \overline{V^{m}}) \quad\text{and} \quad (\widetilde{U},\widetilde{V}) \sim (U_{\rm reg}^{(n)}, V_{\rm reg}^{(n)}),$$
where $m=\dim S, n=\dim (\H)$; moreover, we  have just proved that 
$$(\overline{U^{m}}, \overline{V^{m}}) \sim (\widetilde{U},\widetilde{V});$$
so $(U^{(m)}, V^{(m)})\sim (U_{\rm reg}^{(n)}, V_{\rm reg}^{(n)})$ and this implies
that $(U^{(\infty)}, V^{(\infty)})\sim (U_{\rm reg}^{(\infty)}, V_{\rm reg}^{(\infty)})$.
As $(U_{\rm reg}, V_{\rm reg})$ does not depend on $(U,V)$, we have therefore
$$(U^{(\infty)},V^{(\infty)}) \sim (U_{\text{Schr}}^{(\infty)}, V_{\text{Schr}}^{(\infty)}).$$

\subsection{Proof of Corollary~\ref{Cor}}
\label{SS:ProofCor}
First, observe that the pair $(U_{\text{Schr}}, V_{\text{Schr}})$
is irreducible (that is, there is no non trivial closed subspace of $L^2(R^d)$
which is simultaneously invariant under $U(R^d)$ and $V(R^d)$).
Indeed, let 
$T: L^2(R^d) \to L^2(R^d)$ be a bounded operator which commutes with all 
$U_{\text{Schr}}(a)$'s and all $V_{\text{Schr}}(b)$'s.
 Since $T$ commutes with all $V_{\text{Schr}}(b)$'s and
since $\nabla_\la^{(d)}(R^d)= \widehat{R^d}$ is weak*-dense in $L^\infty(R^d),$ the operator $T$ commutes
with every operator $m_\vfi: L^2(R^d) \to L^2(R^d)$ given by multiplication 
with a function $\vfi\in L^\infty(R^d).$ As is well-known (see e.g. \cite[1.H.6]{BH}),
this implies that $T$ is itself of the form $T=m_\vfi$ for some $\vfi\in L^\infty(R^d).$
 Since $T$ commutes with all $U_{\text{Schr}}(a)$'s, the function $\vfi$ is invariant by 
 the translations $x\mapsto x+a.$ So, $\vfi$
 is constant and  $T$ is a multiple of the identity. Schur's lemma shows that $(U_{\text{Schr}}, V_{\text{Schr}})$
is irreducible.
 
 Next,   let $(U,V)$ be a pair of  unitary representations of $R^d$ on separable Hilbert space $\H,$
satisfying \ref{CCR-la}.  Then, by Theorem~\ref{Theo2}, there exists a unitary isomorphism 
 $$T: \oplus_{i\in \NN} \H_i\to \oplus_{j\in \NN} \K_j,$$
 which intertwines $(U^{(\infty)},V^{(\infty)})$ and  $(U_{\text{Schr}}^{(\infty)}, V_{\text{Schr}}^{(\infty)}),$
 where the $\H_i$'s are copies of $\H$ and the $\K_j$'s are copies of $L^2(R^d).$
 Let $J$ be the set of all $j\in \NN$ for which the map
 $$T_j:= P_j \circ T|_{\H_1}: \H_1\to \K_j$$ 
 is non zero.
 Since $T_j$  intertwines $(U,V)$ and  $(U_{\text{Schr}}, V_{\text{Schr}})$
 and since $(U_{\text{Schr}}, V_{\text{Schr}})$ is irreducible,
 the map $\oplus_{j\in J} T_j: \H_1\to \oplus_{j\in J} \K_j$
 is a unitary isomorphism and sets up an equivalence between
 $(U,V)$ and $(U_{\text{Schr}}^{(|J|)}, V_{\text{Schr}}^{(|J|)}).$

\subsection{A lemma on partitions}
\label{SS:Partition}
 For the proof of Theorem~\ref{Theo1}, we will need  the following general   lemma about partitions.
 
  \begin{lemma}
 \label{Lem-Partition}
  Let $G$ be a second-countable locally compact group
and let $U$ be a compact neighbourhood of the identity in $G$ 
 with $U^{-1}= U.$
We can find  a partition  of $G= \coprod_{i\geq 1}\Omega_i$
 in countably many Borel subsets $\Omega_i^{(n)}$ with the following properties:
 \begin{itemize}
 \item[(P1)] For every $i\geq 1$ and any elements $x,y$ in $\Omega_i,$ we have 
 $y^{-1}x\in U;$
 \item[(P2)] $\Omega_i$ is contained in the closure of the interior ${\rm Int}(\Omega_i)$
 of $\Omega_i,$ for  every $i\geq 1; $ 
 \end{itemize}
  \end{lemma}
  \begin{proof}
  Since $G$  is second-countable, $G$ contains a  countable dense subset $D=\{d_1, d_2, \cdots\}$.
    Let $V$ be an  open  neighbourhood of  the identity with $V^{-1}= V$ and $V^2\subset U.$
 Let  $g\in G;$  since $gV$ is a neighbourhood  of $g$, we have
  $gV\cap D\neq \emptyset$, that is, $g\in DV^{-1}= DV.$ Therefore, we have $G= \bigcup_{i\geq 1} d_iV.$
Set $\Omega_{1}:= d_1\overline{V}$ and, for  define recursively  $\Omega_{i}$ by 
 $$
 \Omega_{i+1}= d_{i+1}\overline{V}\setminus \bigcup_{1\leq j\leq i} \Omega_{j} \tout i\geq 1.
 $$
  It is obvious that $G= \bigcup_{i}\Omega_i$ is a partition  of $G;$
 moreover, for $x,y \in \Omega_i,$ we have 
 $$y^{-1}x \in \overline{V}^{-1} \overline{V}= \overline{V}^2\subset \overline{U}= U$$ and so Condition (P1) is satisfied.
 
 We claim that Condition (P2) is also satisfied. Indeed, 
  ${\rm Int}(\Omega_1)= d_1V$ is dense in $\Omega_{1}.$
 Moreover, for every $i\geq 2,$
 $$\bigcup_{1\leq j\leq i-1} \Omega_{j}= \bigcup_{1\leq j\leq i-1} d_{j} \overline{V}$$
 is closed in $G$; therefore, the set $d_i V \setminus \bigcup_{1\leq j\leq i-1} \Omega_{j}$
 is open in $G$ and  dense in  $ \Omega_{i}.$
 
  \end{proof}
   \subsection{Proof of Theorem~\ref{Theo1}}
\label{SS: ProofTheo1}
 When Condition~\ref{(Iso)} from Theorem~\ref{Theo2} is not satisfied, 
the pairs $(\overline{U^{m}}, \overline{V^{m}})$ and  $(\widetilde{U},\widetilde{V})$ from Proposition~\ref{Pro-Reg} are no longer equivalent. 
We will first   realize    $(\widetilde{U},\widetilde{V})$ on $L^2(\widehat{S}, \H)$ 
 by means of the Fourier transform.
Let $\nu$ be the Haar measure on $\widehat{S}$ for which Plancherel's theorem holds,
that is, for which the  Fourier transform
$\mathcal{F}: L^2(S, \mu\otimes \mu) \to L^2(\widehat{S}, \nu),$ defined by 
$$
\mathcal{F}(f) (\chi)= \int_{S}  \overline{\chi(x,y)} f(x,y)d\mu(x) d\mu(y) \tout \chi\in \widehat{S},
$$
is an isometry. We extend $\mathcal{F}$ to an isometry 
$ L^2(S, \H)\to L^2(\widehat{S},\H),$ which we again denote by $\mathcal{F}$,  by 
$$
\mathcal{F}(F)(\chi)= \int_S \overline{\chi(x,y)} F(x,y) d\mu(x) d\mu(y) \tout F\in L^2(S,\H), \, \chi\in \widehat{S}.
$$
The inverse $\mathcal{F}^{-1}$ of $\mathcal{F}$ is given by 
$$
\mathcal{F}(H)(x,y)= \int_{\widehat{S}} \chi(x,y) H(\chi) d\nu(\chi) \tout H\in L^2(\widehat{S},\H),\, (x,y)\in S.
$$
We transfer the pair  $(\widetilde{U},\widetilde{V}))$ 
to the  pair $(\mathcal{F}(\widetilde{U}),\mathcal{F}(\widetilde{V}))$
on $L^2(\widehat{S},\H)$ given by 
$$\mathcal{F}(\widetilde{U})(a)= \mathcal{F}\widetilde{U}(a) \mathcal{F}^{-1}
\quad\text{and} \quad \mathcal{F}(\widetilde{V})(b)= \mathcal{F}\widetilde{V}(b) \mathcal{F}^{-1}.
$$
Using the translation invariance of $\mu$,  one checks that 
$$
\leqno{(*)}\left\{
\begin{aligned}
\mathcal{F}(\widetilde{U})(a)(H)(\chi_1, \chi_2)&= \chi_1(a) U(a)(H(\chi_1, \chi_2))\\
\mathcal{F}(\widetilde{V})(b)(H)(\chi_1, \chi_2)&= \chi_2(b) V(b)(H(\chi_1, \chi_2))
\end{aligned}
\right.
$$
for all $H \in L^2(\widehat{S}, \H), (a,b)\in S,$ and $(\chi_1, \chi_2)\in \widehat{S}.$

The pair $(U^{(m)},V^{(m)})$  is realized on $L^2(\widehat{S}, \H)\cong \H^{(m)}$ by 
$$
\leqno{(**)}\left\{
\begin{aligned}
U^{m}(a)(H)(\chi_1, \chi_2)&= U(a)(H(\chi_1, \chi_2))\\
V^{m}(b)(H)(\chi_1, \chi_2)&= V(b)(H(\chi_1, \chi_2)).
\end{aligned}
\right.
$$

Recall that $\widehat{S}$, equipped with the topology of uniform convergence on compact subsets
of $S$, is a locally compact group. Moreover, $\widehat{S}$,
 is second-countable, since $R$  is assumed to be second-countable (see e.g. \cite[A.7.3]{BH}).

We claim that  the image of
$\nabla_\lambda^{(2d)}: S \mapsto \widehat{S}$ is dense in $ \widehat{S}.$
Indeed, let $(x,y) \in S$ be such that $\chi(x,y)=1$ for all 
$\chi \in \nabla_\lambda^{(2d)}(S),$ that is, such that
$\la(x\cdot a)\la(y\cdot b)=1$ and hence $\la(x\cdot a)=\la(y\cdot b)=1$ 
 for all $a,b\in R^d.$ By Conditions \ref{Sym} and \ref{Dens}, it follows that
$x=y=0. $ Pontrjagin's duality implies  then that 
the subgroup $\nabla_\lambda^{(2d)}(S)$ is dense in $ \widehat{S}.$

We can now proceed with the proof of Theorem~\ref{Theo1}.
Fix a compact subset $K$ of $R^d$ and a real number $\eps>0.$
We are going to construct a unitary isomorphism $\Phi:L^2(\widehat{S},\H) \to L^2(\widehat{S},\H)$
such that 
$$\sup_{a\in K}\Vert \Phi \mathcal{F}(\widetilde{U})(a)\Phi^{-1} - U^{m}(a) \Vert \leq \eps
\quad\text{and}\quad \sup_{b\in K}\Vert \Phi \mathcal{F}(\widetilde{V})(b)\Phi^{-1} - V^{m}(b) \Vert \leq \eps.
$$
Set $C:= K\times K$ and consider the compact neighbourhood $U$  of $1_S$ (the identity element in  $\widehat{S}$) defined by 
$$
U= \{\chi\in \widehat{S} : \sup_{s\in C} | \chi(s)-1|\leq \eps\}.
$$
Let $\widehat{S}=\coprod_{i\geq 1}\Omega_i$
be a partition   of $\widehat{S}$ in countably many Borel subsets $\Omega_i$
satisfying Conditions (P1) and (P2)  from Lemma~\ref{Lem-Partition}.

Since $\nabla_\lambda^{(2d)}(S)$ is dense in $ \widehat{S}$, it follows 
from  Condition (P2) that 
we can find,  for every $i\geq 1,$ an element  $s_i\in S$ with 
$\nabla_\lambda^{(2d)}(s_i)\in \Omega_i.$
Let $a_i, b_i\in R^d$ be such that  $s_i= (a_i, b_i).$

We define  an    operator 
$$\Phi: L^2(\widehat{S},\H) \to   L^2(\widehat{S},\H)$$
as follows. For $F\in L^2(\widehat{S},\H),$  set
$$
\Phi(F)(\chi)= V(-a_i)U( b_i) (F(\chi)) \tout i\geq 1   \quad\text{and} \quad \chi \in  \Omega_i.
$$
We have
$$
\begin{aligned}
\Vert \Phi(F)\Vert^2 &= \sum_{i\geq 1} \int_{\Omega_i} \Vert V(-a_i)U( b_i) (F(\chi))\Vert^2 d\nu(\chi)\\
&=  \sum_{i\geq 1} \int_{\Omega_i} \Vert F(\chi)\Vert^2 d\nu(\chi)\\
&=\Vert F\Vert^2.
\end{aligned}
$$
Therefore, $\Phi$ is an isometry. Moreover,  $\Phi$ is bijective 
with inverse  defined   by  
$$\Phi^{-1}(F)(\chi)= U(-b_i)V(a_i) (F(\chi))  \tout i\geq 1   \quad\text{and} \quad \chi \in  \Omega_i.$$
So, $\Phi:L^2(\widehat{S},\H)\to L^2(\widehat{S},\H)$ is a unitary isomorphism.
 
Let $s=(a,b)\in S$ and $F\in L^2(\widehat{S},\H).$
 For $i\geq 1$ and $\chi= (\chi_1, \chi_2) \in  \Omega_i,$ on the one hand,
 we have, in view of  Formula $(*)$ for the pair
 $(\mathcal{F}(\widetilde{U}), \mathcal{F}(\widetilde{V}):$
$$
\begin{aligned}
\mathcal{F}(\widetilde{U})(a)(\Phi(F))(\chi)&= \chi_1(a) U(a)(\Phi(F)(\chi))\\
&= \chi_1(a) U(a)(\Phi(F)(\chi)) \\
\end{aligned}
$$
 and similarly 
 $$
\begin{aligned}
\mathcal{F}(\widetilde{V})(b)(\Phi(F))(\chi)&= \chi_2(b) V(b)(\Phi(F)(\chi))\\
&= \chi_2(b) V(b)(\Phi(F)(\chi)); \\
\end{aligned}
$$
on the other  hand,  we have, in view of  Formula $(**)$ for the pair
 $(U^{(m)},V^{(m)}):$ 
$$
\begin{aligned}
\left( \Phi(U^{(m)} (a)F\right)(\chi)&=  \Phi(U(a)(F(\chi))\\
&=V(-a_i)U( b_i)U(a)(F(\chi))\\
&= V(-a_i)U(a)U(b_i)(F(x,y))\\
&= \la(a\cdot a_i) U(a)(V(-a_i)U(b_i)(F(\chi))\\
&= \la(a_i\cdot a) U(a)(\Phi(F(\chi))\\
\end{aligned}
$$
 and similarly 
 $$
\begin{aligned}
\left( \Phi(V^{(m)} (b)F\right)(\chi)&=  \Phi(V(b)(F(\chi))\\
&=V(-a_i)U( b_i)V(b)(F(\chi))\\
&= \la(b_i\cdot b)  V(b)\Phi(F(\chi)).
\end{aligned}
 $$
 Now, by the choice of $s_i=(a_i, b_i),$ we have
$\nabla_\lambda^{(2d)}(s_i)\in \Omega_i$; hence, by Condition (P1),  we have
$\chi \in \nabla_\lambda^{(2)}(s_i)U,$ that is,
 $$
  \sup_{(a,b)\in C} \left| \chi(a,b)- \nabla_\lambda^{(2)}(s_i)(a,b)\right|\leq \eps
  $$
and consequently, using Condition \ref{Sym},
$$\sup_{a\in K} \left| \chi_1(a)-  \la(a_i\cdot a)\right|\leq \eps\quad\text{and} \quad
\sup_{b\in K} \left| \chi_2(b)-  \la(b_i\cdot b)\right|\leq \eps.
$$
We have therefore, for $a\in K,$
$$
\begin{aligned}
\Vert \mathcal{F}(\widetilde{U})(a)\Phi(F)-\Phi (U^{(m)} (a)F) \Vert^2
 &= \sum_{i\geq 1}  \int_{ \Omega_{i}}  |\chi_1(a)- \la(a_i\cdot a)|^2\Vert U(a)(\Phi(F)(\chi))\Vert^2 d\nu(\chi)\\
&=  \sum_{i\geq 1}  \int_{ \Omega_{i}}  |\chi_1(a)- \la(a_i\cdot a)|^2\Vert \Phi(F)(\chi)\Vert^2 d\nu(\chi)\\
&\leq \eps^2 \sum_{i\geq 1}  \int_{ \Omega_{i}} \Vert \Phi(F)(\chi)\Vert^2 d\nu(\chi)=\eps^2 \Vert \Phi(F)\Vert^2\\
&=\eps^2 \Vert F\Vert^2
\end{aligned}
$$
 and, similarly, for $b\in K,$
 $$
\begin{aligned}
\Vert \mathcal{F}(\widetilde{V})(b)\Phi(F)-\Phi (V^{(m)} (b)F) \Vert^2
&\leq \eps^2 \Vert F\Vert^2.
\end{aligned}
$$
So,  $ (U^{(m)}, V^{(m)})$ is approximately equivalent to $(\mathcal{F}(\widetilde{U}),\mathcal{F}(\widetilde{V})),$  
which is obviously equivalent to $(\widetilde{U}, \widetilde{V}))$;
since, by Proposition~\ref{Pro-Reg}, $(\widetilde{U}, \widetilde{V}))$ is equivalent to 
 $(U_{\rm reg}^{(n)}, V_{\rm reg}^{(n)}),$ 
this concludes the proof of Theorem~\ref{Theo1}.
 
  \section{Examples}
\label{S:Examples}
We give a few examples of locally compact rings $R$ and 
unitary characters $\la$ in connection with the various  Conditions~\ref{(Iso)},
\ref{Sym}, and ~\ref{Dens}. 
\begin{enumerate}
 \item Let $\mathbf{k}$ be a  locally compact field,
 that is,  a non-discrete locally compact topological field; it is well-known 
 (see e.g.  \cite[Chap VI, \S 9]{Bourbaki-AC}) that  $\mathbf{k}$ is topologically isomorphic to one of the following fields
  and  that its topology is defined by the corresponding absolute value:
  \begin{itemize}
 \item $\RRR$ or $\CCC$ with the usual absolute
value $|\cdot|.$
\item  a finite extension of the field of $p$-adic numbers
$\Q_p$ with an extension of the $p$-adic absolute value $|\cdot|_p.$
\item the field $\mathbf{F}((X))$  of Laurent series
over a finite field $\mathbf{F}$ with absolute value
$|\sum_{i=m}^{\infty} a_iX^i|= e^{-m}$ with 
$a_m\in \mathbf{F}\setminus \{0\}.$
\end{itemize}

Let $\la\in\widehat{\mathbf{k}}$ be a non trivial character of the additive group of  $\mathbf{k}.$ 
The  map 
$$\nabla_\lambda: \mathbf{k}\mapsto \widehat{\mathbf{k}}, \, a\mapsto (x\mapsto \la(ax))
$$
is an isomorphism of topological groups (see \cite[Chap. II, Theorem 3]{Weil-Book} for a more general result); for the convenience of the reader,  here is a  direct  proof. 
The kernel of $\nabla_\lambda$ is an ideal of $\mathbf{k}$
and hence is trivial, since $\la\neq 1_{\mathbf{k}}.$
By Pontrjagin duality, the image of $\nabla_\lambda$ is a dense subgroup $H$ of $\widehat{\mathbf{k}^d}$
(see the argument  at the beginning of  Subsection~\ref{SS: ProofTheo1}).
We claim that the map $\nabla_\lambda,$ which is obviously continuous, is in fact a homeomorphism between
$\mathbf{k}$ and $H$.  Indeed, assume, by contradiction,
that $\nabla_\lambda^{-1}: H\to \mathbf{k}$ is not continuous. Then, there exists a sequence $(a_n)_n$ in $\mathbf{k}$
and a number $\delta>0$  with $|a_n|>\delta$ for all $n$ such that $\lim_n\nabla_\lambda(a_n)(a)= 1$
uniformly on bounded subsets of $\mathbf{k}$. Since $\la$ is non trivial,
we can find $a\in \mathbf{k}$ with $\la(a)\neq 1.$ 
However, since  $| a a_n^{-1}|\leq \dfrac{|a|}{\delta},$ 
we have
 $$
\la(a)= \la (aa_n^{-1} a_n)= \lim_n\la (aa_n^{-1} a_n)= \lim_n\nabla_\lambda(a_n)(aa_n^{-1})=1
$$
and this is a contradiction. So, $\nabla_\lambda^{-1}$ is continuous
and hence  $H$ is  locally compact; this  implies that it  is a closed
subgroup of $\widehat{\mathbf{k}}$ and, since it is also dense, we have $H= \widehat{\mathbf{k}}.$

\item Let $R= \mathbf{A}$ be the ring of ad\`eles of  $\QQ$.
Recall that   $\mathbf{A}$   is the restricted product  
$\mathbf{A}= \RRR\times \prod_{p\in \P} (\QQ_p, \ZZ_p)$
relative to the subgroups $\ZZ_p$; thus, 
$$\mathbf{A}= \left\{(x_\infty, (x_p)_{p\in \P}) \in \RRR\times \prod_{p\in \P} \QQ_p\, :\, x_p \in \ZZ_p \text{ for almost every } p\in \P \right\},$$
where  $\P$ be the set of primes numbers, $\QQ_p$ is the field of  $p$-adic numbers, 
and $\ZZ_p=\{ x\in \QQ_p \, :  |x|_p\leq 1\}$ is the compact open
subring of $p$-adic integers.

For every finite subset $F\subset \P,$ 
$$\mathbf{A}(F)=  \RRR\times \prod_{p\in F} \QQ_p \times \prod_{p\notin F} \ZZ_p,$$
equipped with the product topology, is  a locally compact ring.
A  ring topology  is defined on $\AA$ on prescribing that 
every  $\mathbf{A}(F)$ is an open subring of $\AA.$  
With this topology, $\AA$ is a locally compact ring.

Let $\la$ be a nontrivial unitary character  of $\AA$.
 An example of  such a character can be obtained as follows.
 Let $\la_\infty \in\widehat{\RRR}$ be given by 
  $\lambda_\infty(x)= e^{-2\pi i x};$  for every 
  $p\in \P,$ let $\la_p\in \widehat{\QQ_p}$ be given by 
 $\la_p(x)= \exp(2\pi i  \{x\}),$ where $\{x\}= \sum_{j=m}^{-1} a_j p^j $ denotes the ``fractional part" of a 
 $p$-adic number $x= \sum_{j=m}^\infty a_j p^j$ for integers $m\in \ZZ$ and $a_j \in \{0, \dots, p-1\}.$
 Define now $\la\in \widehat{\AA}$ by 
 $$
 \la(x_\infty, (x_p)_{p\in \P})= \la_\infty(a_\infty)\prod_{p\in \P} \la_p(x_p) \tout  (x_\infty, (x_p)_{p\in \P})\in \AA
 $$
 (the product of the right hand side of the formula is well-defined, since
 $x_p \in \ZZ_p$  for almost every $ p$ and since $ \la_p$ is trivial on $\ZZ_p$ for
 every $p$).
 Then $\la$ satisfies   Condition~\ref{(Iso)}, that is, 
$$
 \AA\mapsto \widehat{\AA}, a\mapsto (b\mapsto \la(ba))
$$
is a topological  isomorphism (see \cite[Chap. IV, Theorem 3]{Weil-Book}).

 \item Let $\mathbf{k}$ be a local field.
 For an  integer $n\geq 1,$ let $R=M_n(\mathbf{k})$ be the ring of $n\times n$-matrices over $\mathbf{k};$
it is a locally compact ring for the topology obtained by identifying  the additive group
of  $M_n(\mathbf{k})$ with $\mathbf{k}^{n^2}.$
Let $\la_0\in\widehat{\mathbf{k}}$ be a non trivial character of  $\mathbf{k}.$ 
Let $\la\in\widehat{M_n(\mathbf{k})}$ be defined by 
$$
\la(A)= \la_0(\Tr(A)) \tout A\in M_n(\mathbf{k}),
$$
where $\Tr(A)\in \mathbf{k}$ is the trace of $A.$ 
Since  $\Tr(AB)=\Tr(BA)$ for all $A,B,$ the character $\la$ satisfies  the symmetry
Condition~\ref{Sym}. It satisfies  Condition~\ref{(Iso)} as well.  Indeed,
in view of  Condition~\ref{Sym}, the kernel of the map 
$$\nabla_\lambda: M_n(\mathbf{k})\mapsto \widehat{M_n(\mathbf{k})}, A\mapsto (X\mapsto \la(XA))
$$
is a  two-sided ideal of  $M_n(\mathbf{k})$; since $M_n(\mathbf{k})$
is a simple ring and   $\la$ is non trivial, it follows that $\nabla_\lambda $ is injective.

On the one hand, viewing   the additive group
of  $M_n(\mathbf{k})$ as $\mathbf{k}^{n^2}, $
Item (1) above shows that every $\chi \in \widehat{M_n(\mathbf{k})}$
is of the form $X\mapsto \la (\vfi(X))$ for some 
linear form $ \vfi\in \Hom_{\mathbf{k}} (M_n(\mathbf{k}), \mathbf{k})$
on $M_n(\mathbf{k}).$
On the other hand, it is well-known that every such $\vfi$
can be written as the map
$X\mapsto \Tr(XA)$ for  some $A\in M_n(\mathbf{k}).$
This shows that  the continuous injective map $\nabla_\lambda$ is surjective and is hence a topological isomorphism.

\item The previous example can be generalized as follows to simple central algebras over local fields;
for general  facts concerning these algebras, see \cite[Chap IX]{Weil-Book}).
 
 Let $\mathbf{k}$ be a local field. Let $\mathcal{A}$ be a finite-dimensional central simple algebra over $\mathbf{k}.$ 
 Then $\mathcal{A}$  is isomorphic to a   matrix algebra $M_n(\mathbf{D})$  over some
central division algebra $\mathbf{D}$ over $\mathbf{k}$ for some $n\geq 1.$ 
 For instance, if $\mathbf{k}= \RRR,$ then  $\mathbf{D}$ is  necessarily 
 is isomorphic to $\RRR$ or to $\HH,$
  where $\HH$ is the $4$-dimensional division algebra of the usual real quaternions.
Let $\tau  \in \Hom_{\mathbf{k}} (\mathcal{A}, \mathbf{k})$ be the so-called \emph{reduced trace} on $\mathcal{A}$,
which is defined as follows:  let  $\mathbf{K}$ be the algebraic closure
of ${\mathbf{k}}$; then there exists an algebra isomorphism
$F: \mathcal{A}\otimes_{\mathbf{k}} \mathbf{K} \to M_r({\mathbf{k}}),$
for some $r$ and one defines $\tau(a)= \Tr(F(a))$ for $a\in \mathcal{A}.$
 
 Let   $\la_0$ be a non trivial unitary character of $\mathbf{k}$
and define  $\la\in\widehat{\mathcal{A}}$  by 
$$
\la(a)= \la_0(\tau(a)) \tout a\in \mathcal{A}.
$$
Then  $\la$ satisfies   Conditions~\ref{Sym} and ~\ref{(Iso)}.
For instance, when $\mathcal{A}$ is the algebra of real quaternions
$$
\HH= \{t+u i+ vj+ wk\, : \, t, u,v, w \in \RRR\},
$$
we have $\tau(a)= 2t$ for  $a= t+u i+ vj+ wk$
and so, taking $\la_0 = e^{2\pi \cdot}, $ we have 
$$
\la (a)= e^{4\pi i t} \tout a= t+u i+ vj+ wk\in \HH.
$$

\item Let $R$ be a finite field or a direct sum of finite fields. 
It can  easily be checked that every non trivial unitary character
$\la$ of $R$ satisfies   Condition~\ref{(Iso)}.
Then same holds when $R$ is the ring $\ZZ/n\ZZ$ of integers 
modulo $n\in \NN^*.$ 

\item For $R=\ZZ$, there exists \emph{no}  $\la\in \widehat{\ZZ}$ which satisfies
Condition~\ref{(Iso)}. Indeed, every $\la\in \widehat{\ZZ}$ is of the form 
$$
\la_\theta (k) \, = \, e^{2 \pi i \theta k}
\tout k\in \ZZ,
$$
for a uniquely determined  real number $\theta \in [0,1)$. In particular,
$\widehat{\ZZ}$ is uncountable and, since $\ZZ$ is countable,
$\nabla_\lambda: \ZZ\to \widehat{\ZZ}$ cannot be surjective.
Observe, however, that Condition~\ref{Dens} is satisfied by $\la_\theta$ 
for every irrational $\theta.$ 
\item  Similarly to the previous Item,  if $R$  any countable  infinite ring (equipped with the discrete topology), then there exists no  $\la\in \widehat{R}$ which satisfies Condition~\ref{(Iso)}. 
\item Let $R= \ZZ_p$ be the ring of  $p$-adic integers for  a prime number  $p.$
Then,  since $\ZZ_p$ is compact and infinite, its dual group $\widehat{\ZZ_p}$ 
is a discrete  infinite group; in fact, the unitary character  $\la_p\in \widehat{\QQ_p}$
from Item (2) induces an isomorphism 
$ \QQ_p/\ZZ_p \cong\widehat{\ZZ_p},$ given by 
$a+ \ZZ_p \mapsto \nabla_{\la_p}(a)|_{\ZZ_p}$ 
So, there exists no  $\la\in \widehat{\ZZ_p}$ which satisfies
 Condition~\ref{Dens} and, let alone, Condition~\ref{(Iso)}.
 However, one can easily determine the pairs of unitary representations $(U, V)$ 
 of $\ZZ_p$ on a  Hilbert space $\H$ satisfying \ref{CCR-la} for any  $\la\in \widehat{\ZZ_p}$;
 indeed, assume by simplicity that $(U,V)$ is irreducible.
 The sequence of ideals  $J_n= p^n\ZZ_p$ ($n\geq 1$)  
 is a basis of neighbourhoods in $0$ in $\ZZ_p.$ So, there exists $N\geq 1$ such that 
 $\la(J_N)=1$ and  such that the space $\H^{J_N}$ of invariant vectors under $U(J_N)$ and $V(J_N)$
 is non zero (see \cite[1.1.5]{BHV}); by irreducibilty, we have then $\H= \H^{J_N}.$
 This shows that $\la$ and  $(U, V)$ factorize to a character 
 $\overline{\la}$ and a pair $(\overline{U}, \overline{V})$
 of unitary representations of the finite  quotient ring $\ZZ_p/J_N\cong \ZZ/ p^N \ZZ $ 
 satisfying  $\rm{(CCR)}_{\overline{\lambda}}.$
 \item The conclusion of the previous Item can be  generalized  to any compact ring.
 Indeed, let $R$ be such a ring (commutative or not). Then
 $R$ is a profinite ring, that is,  $R$ is topologically isomorphic to a projective limit  $\varprojlim_i R_i$ 
 of finite rings $R_i$ (see \cite[Proposition 5.1.2]{Ribes-Zaleskii}).
 As in the previous item, the study of pairs of unitary representations $(U, V)$ 
 of $R$ satisfying \ref{CCR-la} for any  $\la\in \widehat{R}$
 can be reduced to the study of such pairs for the finite rings $R_i.$
 
 \end{enumerate}

 \section{Heisenberg groups}
 \label{Heisenberg}
 We rephrase our  results in terms of representations of Heisenberg groups.
  Let $R$ be  a  (not necessarily commutative) locally compact ring with unit 
and let $d\geq 1$ be an integer.
The \emph{Heisenberg group} of dimension $2d+1$ over $R$ is 
is defined as the group $H_{2d+1}(R)$ of $(2d+1)\times (2d+1)$-matrices with coefficients  in $R$ of  the 
form 
\[
 m(a,b, c)=m(a_1, \dots, a_d, b_1, \dots, b_d, c):=
\left(
\begin{array}{cccccc}
 1&a_1&\dots &a_{d}&c\\
 0&1&\dots&0&b_1\\
 \vdots & \ddots &\ddots& \vdots&\vdots \\
 0&0&\dots&1&b_d\\
0&0&\dots&0&1\
\end{array}
\right),
\]
The group law on  $H_{2d+1}(R)$ is given by 
$$
m(a,b, c) m(a', b', c') \, = \, m(a+ a', b+b',  c+c' + a\cdot b'),
$$
for  $a, b, a', b'\in R^d$ and $c, c'\in R.$
Equipped with the  topology induced from $M_{2d+1}(R),$ 
 the group $H_{2d+1}(R)$ is a locally compact group.
Its  center is 
$$C= \{m(0,0, c)\, : \, c\in R\} \cong \, R.$$

Let $\pi: G\to U(\H)$
and $\pi': G\to U(\H')$ be two unitary representations  of  a locally compact group  $G$ on Hilbert spaces $\H$ and $\H'$.
Recall that $\pi$ and $\pi'$  are equivalent
if there exists a unitary isomorphism $\Phi: \H\to \H'$
such that 
$$\Phi \pi(g) \Phi^*= \pi'(g)\tout g\in G.$$
We  say that  $\pi$ and $\pi'$ are 
\emph{approximately equivalent} if there exists
 a sequence  $(\Phi_n)_{n}$ of unitary  isomorphisms $\Phi_n: \H_1\to \H_2$
 such that, for every compact subset $K$ of $G,$ we have
$$ \lim_{n\to \infty} \sup_{g\in K} \Phi_n \pi(g) \Phi_n^*- \pi'(g)\Vert=0.$$
Approximate equivalence is an equivalence relation on, say, the set of unitary representations
of $G$ on separable Hilbert spaces.

Fix  $\lambda\in \widehat{R}$.
Let $(U,V)$ be a  pair of unitary representations   of  $R^d$  on a Hilbert space $\H$   satisfying \ref{CCR-la},
For $m(a,b,c)\in H_{2d+1}(R),$ define a unitary operator $\pi(m(a,b,c))$ on $\H$
by 
$$
 \pi(m(a,b,c)=\la(c) V(b)U(a).
 $$
 One checks immediately that $\pi: H_{2d+1}(R)\to U(\H)$
 is a unitary representation of $H_{2d+1}(R);$ observe that 
  $\pi$ has $\la$ as central character, that is, $\pi(m(0,0,c))= \la(c){\rm Id}_\H$
 for $c\in R$.
 If $(U',V')$ is another pair which is equivalent (respectively, approximately equivalent) to $(U,V)$,
 then the corresponding representation $\pi'$ is equivalent (respectively, approximately equivalent) to $\pi.$

 Conversely, let $\pi: H_{2d+1}(R)\to U(\H)$ 
 be  a unitary representation of $H_{2d+1}(R)$ which 
 has $\la$ as central character.
 Define 
 $$
 U(a)= \pi(m(a, 0,0)) \quad \text{and}  \quad V(b)= \pi(m(0, b,0))
 $$
 for $a, b\in R^d.$ Then   $(U, V)$ is a pair of unitary representations   of  $R^d$  on $\H$   satisfying \ref{CCR-la}.
 Moreover, if  $\pi'$ is another representation 
 of  $H_{2d+1}(R)$ which is equivalent (respectively, approximately equivalent) to $\pi$,
 then $\pi'$ has  $\la$ as central character
 and the corresponding pair $(U',V')$  is equivalent (respectively, approximately equivalent) to $(U,V)$.

 Summarizing,  the map  $(U, V) \mapsto  \pi$ described above
induces a    bijection  between
\begin{itemize}
\item the set  of equivalence (respectively, approximate equivalence)  pairs of unitary representations   of  $R^d$
 satisfying \ref{CCR-la}, and 
\item the set  of   equivalence (respectively, approximate equivalence)  classes of   unitary representations of $H_{2d+1}(R)$ which have  $\la$ as central character,
\end{itemize}
The representation $\pi_{\text{Schr}}$ of $H_{2d+1}(R)$ corresponding to the 
Schr\"odinger pair $(U_{\text{Schr}}, V_{\text{Schr}})$ with central character $\la$ is given
on $L^2(R^d, \mu)$ by 
$$
\pi_{\text{Schr}}(m(a,b,c))f)(x)= \la(c)\lambda(x\cdot b)f(x+a) \quad f \in L^2(R^d, \mu),
$$
where $\mu$ is a Haar measure on $(R^d,+).$

We can now reformulate Theorems~\ref{Theo2} and ~\ref{Theo1} in terms of 
representations of $H_{2d+1}(R).$
\begin{thmx}
\label{Theo2-bis}
(\textbf{Stone-von Neumann Theorem}) 
Let  $R$ be  a  unital   second-countable locally compact ring and let  $\lambda\in \widehat{R}$
 satisfying Conditions~\ref{Sym} and ~\ref{(Iso)}.
For  $d\geq 1,$  let $\pi$ be a unitary representation of $H_{2d+1}(R)$ 
 on a separable Hilbert space  with central character $\la.$
Then  $\pi^{(\infty)}$  is equivalent to $\pi_{\text{Schr}}^{(\infty)}$.
\end{thmx}
\begin{thmx}
\label{Theo1-bis}
Let  $R$ be  as Theorem~\ref{Theo2-bis}, and  let  $\lambda\in \widehat{R}$
satisfying  Conditions~\ref{Sym} and ~\ref{Dens}. For  $d\geq 1,$
 let $\pi$ be a unitary representation of $H_{2d+1}(R)$ 
 on a separable Hilbert space with central character $\la.$
 Then  $\pi^{(\infty)}$  is approximately equivalent to $\pi_{\text{Schr}}^{(\infty)}$.
 \end{thmx}

\begin{remark}
\label{Rem-Reg}
A crucial role in the proofs of  our results was played by 
the regular pair $(U_{\rm reg}, V_{\rm reg})$  introduced just before Proposition~\ref{Pro-Reg}.
 The corresponding representation   $\pi_{\rm reg}$ of $H_{2d+1}(R)$
 is defined on  the Hilbert space $ L^2(R^d\times R^d, \mu\otimes \mu),$
by 
$$
(\pi_{\rm reg}(m(a,b, c)\xi)(x,y)=  \la(c)\lambda(x\cdot b) \xi(x+a,y+b),
$$
for all $\xi \in L^2(R^d\times R^d)$ and $x, y\in R^d.$ 
It can be checked that $\pi_{\rm reg}$ is equivalent to the  representation
$\ind_{C}^{H_{2d+1}(R)} \la$  induced by the character $\la$ of the centre $C.$

To a  pair $(U,V)$    of unitary representations 
of  $R^d$  on $\H$ satisfying \ref{CCR-la},  we associated  a new pair $(\widetilde{U},\widetilde{V})$   
 on   $L^2(R^d\times R^d, \H)$.
 If $\pi$ is the representation of $H_{2d+1}(R)$ corresponding to $(U,V)$,
 the representation $\widetilde{\pi}$ of $H_{2d+1}(R)$ corresponding to 
 $(\widetilde{U},\widetilde{V})$ is given by 
 defined  by 
$$
(\widetilde{\pi}(m(a, b, c)F)(x,y)= \la(c) V(b) U(a)(F(x+a,y+b)) $$
for all $F \in L^2(R^d\times R^d, \H)$ and $x, y\in R^d.$
It can be checked that 
$\widetilde{\pi}$ is equivalent to the tensor product representation $\rho \otimes \pi,$
where $\rho=\ind_{C}^{H_{2d+1}(R)}1_C $ is the regular representation of $H_{2d+1}(R)/C\cong R^d\times R^d$ lifted to $H_{2d+1}(R).$
Fell's  absorption principle alluded to before Proposition~\ref{Pro-Reg} is that 
$\rho \otimes \pi$ is equivalent to a multiple of $\ind_{C}^{H_{2d+1}(R)} \la.$
\end{remark}


\bibliographystyle{amsalpha}

\end{document}